\title{A short proof of the `thumbtack lemma'}
\author{
        Adam Harris \\
        Oxford University
}
\date{\today}

\documentclass[11pt,a4paper]{article}

\usepackage{amssymb ,amsfonts, amsmath, amscd, amsthm, verbatim, natbib, graphics, tikz}
\usepackage{graphics}                 
\usepackage{color}                    
\usepackage{hyperref}                 
\usepackage{appendix}
\usepackage{fullpage}
\usepackage[small,nohug,heads=vee]{diagrams}
\diagramstyle[labelstyle=\scriptstyle]

\newtheorem{thm}{Theorem}[section]
\newtheorem{cor}[thm]{Corollary}
\newtheorem{lem}[thm]{Lemma}
\newtheorem{pro}[thm]{Proposition}

\theoremstyle{remark}

\theoremstyle{remark}

\theoremstyle{definition}

\def\G{\mathbb{ G}}

\def\N{\mathbb{ N}}

\def\Z{\mathbb{ Z}}
\def\Q{\mathbb{ Q}}

\def\C{\mathbb{ C}}

\def\*{{}^{*}}

\def\cF{\mathcal{F}}

\def\ra{\rightarrow}

\def\lora{\longrightarrow}
\def\gl2q{\text{\textnormal{GL}}_2^+(\Q)}
\def\ggl2{\text{\textnormal{GL}}_2}

\def\sl2{\text{\textnormal{SL}}_2}
\def\psl2{\text{\textnormal{PSL}}_2}
\def\sv2{\text{SZ}_2}
\def\z2{\text{Z}_2}
\def\Hom{\text{\textnormal{Hom}}}
\def\pgl2{\text{PGL}_2}

\def\m2z{\text{M}_2(\Z)}

\def\Spec{\text{\textnormal{Spec}}}
\def\Gal{\text{\textnormal{Gal}}}

\def\g0n{\Gamma_0(N)}

\include{head}

\begin{document}

\maketitle

\begin{abstract}
We give a short proof of the main algebraic result of \cite{zilber2006covers}, also known as the `thumbtack lemma'.
\end{abstract}

\section{Introduction}

In this note we give a short proof of the main algebraic result (Theorem 2) in \cite{zilber2006covers}. This is what has previously been called the $n=0$ and $n=1$ thumbtack lemmas which we call `arithmetic homogenity' and `geometric homogeneity' respectively. The exposition given here is hopefully written in a style more accesible to algebraic geometers, in the familiar setting of the pro-\'etale cover. However, the pro-\'etale cover here is purely a conceptual device, and the actual proof does not require any of this machinery. \newline

Let $F$ be a field of characteristic 0, and consider the multiplicative group of a field $\G_m(F)$ (which we will also sometimes denote by $F^{\times}$). Model theoretically we may think of $\G_m(F)$ as a definable set in $\cF:=\langle F,+, \cdot ,0,1 \rangle$
as $$\G_m(F):=\{(x,y) \in F^2 \ | \ xy=1 \}\subset F^2,$$
or we may also think of $\G_m(F)$ as the $F$-points of the scheme $\Spec(\Q[X,Y] / (XY-1))$
It is well known (e.g. by Riemann-Hurwitz) that finite \'etale covers of $\G_m(F)$ are of the form
$$\G_m(F) \overset{x \mapsto x^n}{\lora} \G_m(F)$$
and that the \'etale automorphisms of this cover are given by multiplications by $n$'th roots of unity.

For simplicity, we consider all fields in the remained of this note as embedded in $\C$. Let $$\hat{p} : \hat{U} \ra \G_m(\C)$$ be the pro-\'etale cover (see for example \cite{milne1998lectures}). We will be interested in the geometric \'etale fundamental group 
$$\pi_1^{et}:=\pi_1^{et}(\G_m(\C)) \cong \hat{\Z}, $$
and we note that  the (geometric) \'etale fundamental group is unchanged when taking an extension of algebraically closed fields i.e. for $F$ embedded in $\C$,  $\pi_1^{et}(\G_m(\bar{F}))= \pi_1^{et}(\G_m(\C))$. In particular, the content of this note is regarding certain Galois representations into $\pi_1^{et}$. \newline

Given $x \in \G_m(F)$, the fibre $\hat{p}^{-1}(x) = t$ is a right $\pi_1^{et}$-torsor. Here, the fibre $\hat{p}^{-1}(x)$ is just a compatible sequence $(x^{1/n})_n$ of $n$'th roots of $x$, and $\pi_1^{et}$ acts on the fibre via multiplication by a compatible sequence of roots of unity. There is also a left action of $G_F:=\Gal(\bar{F} / F)$ on $\pi_1^{et}$ (which we may think of coming via its action on the roots of unity), and the actions of $\pi_1^{et}$ and $G_F$ on the fibre $t$ are compatible i.e. for $u \in t$ we have
$$\sigma(u.g) = \sigma(u).\sigma(g).$$
Given such a torsor $t=\hat{p}^{-1}(x)$, we may produce a cocycle
$$\rho_x : G_F \ra \pi_1^{et}$$
as follows:\newline

Given $\sigma \in G_K$ and $u \in t$, since $t$ is a $\pi_1^{et}$-torsor there is $g_{\sigma} \in \pi_1^{et}$ such that $\sigma(u) = u.g_{\sigma}$. Now since the actions of $G_K$ and $\pi_1^{et}$ are compatible, we have
$$(\sigma_1 \sigma_2)(u) = \sigma_1(\sigma_2(u)) = \sigma_1(u.g_{\sigma_2}) = \sigma_1(u) . \sigma_1(g_{\sigma_2}) = u.g_{\sigma_1}\sigma_1(g_{\sigma_2})$$
so that
$$\sigma \mapsto g_{\sigma}$$
is a cocycle.\newline

We can see concretely what is happening, as if we take $x \in K$ then for $\sigma \in G_K$ we have
$$\sigma(u) = \sigma(x^{1/n})_n = (x^{1/n})_n.(\zeta_n)_n$$
for some compatible sequence of roots of unity $(\zeta_n)_n$. So here, the action of $ g_{\sigma} \in \pi_1^{et} $ is given by componentwise multiplication by $(\zeta_n)_n$ and we can think of $\rho_x$ as a profinite version of the traditional Kummer map (arising as the connecting homomorphism in the Kummer exact sequence - see the proof of \ref{kummer}) i.e.
\begin{align*}
 \rho: \G_m(F) &\ra H^1(G_F , \pi_1^{et})\\
 x & \mapsto \left[ \sigma \mapsto \frac{\sigma(x^{1/n})_n}{(x^{1/n})_n} \right]
\end{align*}
(this is well defined since clearly any two elements of a fibre $\hat{p}^{-1}(x)$ differ by a coboundary). Now if the roots of unity are in $K$, then $H^1(G_F,\pi_1^{et}) = \Hom(G_F , \pi_1^{et})$ and we obtain a much clearer picture of what $G_F$ looks like (i.e. we are in the realms of Kummer theory, see proof of \ref{kummer}).\newline

For a tuple $(x_1,...,x_n): =\bar{x} \in F^n$ and the correspondiing tuple of fibres $$\hat{p}^{-1}(\bar{x}):=(\hat{p}^{-1}(x_1),...,\hat{p}^{-1}(x_n))$$ we define
\begin{align*}
\rho_{\bar{x}}:G_F & \ra (\pi_1^{et})^n.
\end{align*}
componentwise.

 We are going to prove the following two theorems, and the proof is based on \cite[V \S 4]{lang1979elliptic}:

\begin{thm}[Arithmetic homogeneity]\label{thumb0}Let $K$ be a number field and $\bar{a} \subset K^{\times}$, a multiplicatively independent tuple. Then the image of the continuous homomorphism
$$\rho_{\bar{a}} : \Gal(K(\mu, \hat{p}^{-1}(\bar{a})) / K(\mu )) \hookrightarrow \pi_1^{et}(\C^{\times})^r$$
 is open.
\end{thm}

\begin{thm}[Geometric homogeneity]\label{thumb1}Let $K \subset \C$ be an algebraically closed field, and $\bar{a} \subset \C-K$ multiplicatively independent. Then the image of the continuous homomorphism
$$\rho_{\bar{a}} : \Gal(K( \hat{p}^{-1}(\bar{a})) / K(\bar{a})) \hookrightarrow \pi_1^{et}(\C^{\times})^r$$
 is open.
\end{thm}

In fact, geometric homogeneity falls out of the proof of arithmetic homogeneity, and is easier since no cohomology is needed to deal with the roots of unity. The main results (Theorems \ref{thumb0} and \ref{thumb1}) of this note, along with the elementary fact that
$$\Gal(\Q(\mu) / \Q) \cong \hat{\Z}^{\times},$$
and the main result of \cite{bays2012quasiminimal}, imply the main categoricity result (Theorem 1) of \cite{zilber2006covers}.

\section{The proof}

A subgroup of a profinite group is open iff it is closed and of finite index. To prove that the image is open in the profinite group 
$$\pi_1^{et} \cong \hat{\Z} \cong \prod_{l} \Z_l$$ we split the proof into a `horizontal' and a `vertical' result:

\begin{lem}\label{horiz}[Horizontal openness]
The image of the continuous homomorphism
$$\rho_{l^{\infty}} : \Gal(K(\mu_{l^{\infty}},\hat{p}_l^{-1}(\bar{a})) / K(\mu_{l^{\infty}}))  \hookrightarrow (\pi_1^{et})_l( \G_m(\C))^r$$
is surjective for almost all primes $l$.
\end{lem}

\begin{lem}\label{vert}[Vertical openness]
The image of the continuous homomorphism
$$\rho_{l^{\infty}} : \Gal(K(\mu_{l^{\infty}},\hat{p}_l ^{-1}(\bar{a})) / K(\mu_{l^{\infty}}))  \hookrightarrow (\pi_1^{et})_l( \G_m(\C))^r$$
is open for all primes $l$.
\end{lem}
For the definition of the $l$-adic fibre $ \hat{p}_l^{-1}(\bar{a})$, just take all compatible sequences of $l^{n}$th roots of $\bar{a}$ for all $n$ and for the definition of $\rho_{l^{\infty} , \bar{a}}$ just choose one of them and proceed as before.

\begin{lem}\label{freeab} \cite[2.1]{zilber2006covers}
Let $K$ be a finitely generated extension of $\Q$. 
Then
$$K^{\times} \cong A \times \left( K \cap \mu \right)^{\times} $$
where $A$ is a free abelian group. If $L$ is algebraically closed, then
$$ (KL)^{\times} \cong B \times L^{\times} $$
where $B$ is a free abelian group.
\end{lem}

Let $K$ be field which is either a number field or the compositum of a finitely generated extension of $\Q$ with an algebraically closed field and $\bar{a}$ a multiplicitively independent tuple coming from the free abelian part of $K^{\times}$ (as in \ref{freeab}). Let $\Gamma$ be the multiplicitive subgroup of $K^{\times}$ generated by $\bar{a}$, and $\Gamma'$ the division group of $\Gamma$ in $K^{\times}$ i.e.
$$\Gamma':=\{x \in K^{\times} \ | \ x^n \in \Gamma \textrm{ for some } n \in \N \}.$$

\begin{cor}\label{corfingen}
Let $K$ be a finitely generated extension of $\Q$ and $\Gamma$ a finitely generated subgroup of $K^{\times}$. Then $[\Gamma':\Gamma]$ is finite.
\end{cor}

At this point, we make the observation that if $n$ is prime to the index $[\Gamma':\Gamma]$, then
$$\Gamma^n = \Gamma \cap K^{\times^n}.$$

\begin{thm}\label{kummer}[Kummer theory]
Let $L$ be a field containing the $n^{th}$ roots of unity $\mu_n$ and $\Gamma$ a finitely generated subgroup of $L^{\times}$. Then
$$\Gal(L( \Gamma^{1 / n}) / L) \cong \Gamma /  \Gamma \cap L^{\times^n} .$$
\end{thm}

\begin{proof}
We start with the `Kummer exact sequence' of $G:=\Gal(L(\Gamma^{1/n}) / L)$-modules
$$\mu_n \lora L(\Gamma^{1/n})^{\times} \overset{x\mapsto x^n}{\lora} L(\Gamma^{1/n})^{\times^n},$$
and we take cohomology to get a long exact sequence
$$H^0(G,\mu_n)\ra H^0(G,L(\Gamma^{1/n})^{\times}) \overset{x\mapsto x^n}{\ra} H^0(G,L(\Gamma^{1/n})^{\times^n})  \overset{\delta}{\ra}  H^1(G, \mu_n) \ra H^1(G, L(\Gamma^{1/n})^{\times}). $$
$G$ acts trivially on $L$ and therefore on $\mu_n$, and $H^1(G, L(\Gamma^{1/n})^{\times})=0$ by Hilbert's theorem 90, so we get an exact sequence
$$\mu_n \lora L^{\times}  \overset{x\mapsto x^n}{\lora} L^{\times}\cap L(\Gamma^{1/n})^{\times^n}   \overset{\delta}{\lora}  \Hom(G,\mu_n) \ra 0.$$
Now we note that $G$ is a finite abelian group of exponent $n$, and so is isomorphic to its character group $\Hom(G,\mu_n)$, giving
$$\Gal(L(\Gamma^{1/n}) / L) \cong \left( L^{\times} \cap L(\Gamma^{1/n})^{\times^n} \right) /L^{\times^n} =\Gamma L^{\times^n} / L^{\times^n} \cong \Gamma /  \Gamma \cap L^{\times^n} .$$
\end{proof}

\subsection{Horizontally}

\begin{pro}\label{arith2}\cite[V 4.2]{lang1979elliptic}
Let $K$ be a number field, $n$ be coprime to $2[\Gamma':\Gamma]$ and
$$\Gal(K(\mu_n) / K) \cong (\Z / N \Z)^{\times}.$$
Then 
$$\Gamma \cap K(\mu_n)^{\times^{l^n}} = \Gamma^n.$$
\end{pro}

\begin{proof}
Suppose not. Then for some prime $p|n$ there is $\alpha \in \Gamma$ such that
$$\alpha = \beta^p \ \ , \ \ \beta \in K(\mu_n) - K$$
(where $\beta \notin K$ since $n$ is coprime to $2[\Gamma':\Gamma]$). Now since $\beta \notin K$, the polynomial $X^p-\alpha$ is irreducible over $K$ and so $\beta$ has degree $p$ over $K$. But we have
$$[K(\mu_p) :K]=p-1$$
so $\beta$ has degree $p$ over $K(\mu_p)$ also. The Galois extension
$$K(\mu_p, \beta) / K$$
is non-abelian and therefore cannot be contained in $K(\mu_n)$, since the extension $K(\mu_n) / K$ is abelian.
\end{proof}

\begin{pro}
Let $K$ be a finitely generated extension of $\Q$. Then 
$$\Gal(K(\mu) / K)$$
is isomorphic to an open subgroup of $\hat{\Z}^{\times}$.
\end{pro}
\begin{proof}
It is an elementary result that
$$\Gal(\Q(\mu) / \Q) \cong  \hat{\Z}^{\times}$$
and the result follows almost immediately, remembering that open is equivalent to closed and finite index.
\end{proof}

\begin{proof} [Proof of Lemma \ref{horiz}]
Since $\Gal(K(\mu) / K)$
is isomorphic to an open subgroup of $\hat{\Z}^{\times}$, there is $l_0$ such that if $l \geq l_0$ then
$$\Gal(K(\mu_{l^n}) / K) \cong (\Z / l^n \Z)^{\times}$$
for all $n$. By Kummer theory (\ref{kummer})
$$\Gal(K(\mu_{l^n}, \Gamma^{1 / l^n}) / K(\mu_{l^n})) \cong \Gamma /  \Gamma \cap K(\mu_{l^n})^{\times^{l^n}} $$
but by \ref{corfingen}, the index $[\Gamma':\Gamma]$ is finite and if $l \geq l_0$ is coprime to $2[\Gamma':\Gamma]$ then by \ref{arith2} we have
$$\Gamma \cap K(\mu_{l^n})^{\times^{l^n}} = \Gamma^{l^n}$$
and so 
$$\Gal(K(\mu_{l^n}, \Gamma^{1 / l^n}) / K(\mu_{l^n})) \cong \Gamma /  \Gamma^{l^n} \cong (\Z / l^n \Z)^r. $$
\end{proof}

\subsection{Vertically}

\begin{lem}[Sah's Lemma]
Let $M$ ba a $G$-module and let $\alpha$ be in the centre $Z(G)$. Then $H^1(G,M)$ is killed by the endomorphism
$$x \mapsto \alpha x-x$$
of $M$. In particular, if $f$ is a 1-cocycle and $g \in G$ then
$$(\alpha -1) f(g) = (g-1)f(\alpha).$$
\end{lem}

Note that we presented Sah's lemma in its familiar additive notation, but we will now apply it in multiplicative notation.

\begin{pro} \label{sahcor}
Let $a \in K^{\times}$ and suppose that $\rho_{l^m,a}(h)=0$ for all  $h \in H:= \linebreak \Gal(K(\mu_{l^m} ,a^{1/l^m}) / K(\mu_{l^m}))$.  Then there is a constant $\lambda (K,l) \in \N$ (independent of $m$) such that $a^{\lambda} \in K^{\times^{l^m}}$.
\end{pro}

\begin{proof}
If $b$ is an $l^m$-th root of $a$, and $\rho_{l^m,a}(h)=0$ for all $h \in H$, then $b$ is in $K(\mu_{l^m})$ since it is fixed by everything in $H$ and the map
$$g \mapsto \frac{g b}{b}$$
is a cocycle from $G:=\Gal(K(\mu) / K)$ to $\mu_{l^m}$. $G$ is isomorphic to an open subgroup of $\hat{\Z}^{\times}$, and units in $\Z_l$ are those which aren't divisible by $l$, so there is $\alpha \in G$ such that $\alpha$ acts on $\mu$ as raising to the power $\lambda:=l^{m_0}+1$ for some $m_0$. Now by Sah's lemma there exists $\zeta \in \mu_{\lambda}$ such that
$$ \left( \frac{g b}{b} \right)^{\lambda} = \frac{g \zeta}{\zeta}.$$
Now if we let $\eta^\lambda = \zeta$, then we have
$$ \left( \frac{g b}{b} \right)^\lambda = \frac{g \zeta}{\zeta} = \frac{g \eta^\lambda}{\eta^\lambda} = \left(\frac{g \eta}{\eta}\right)^\lambda,$$
so $c:=(b. \eta^{-1})^\lambda $ is fixed under all $g \in G$ and is therefore in $K^{\times}$, and $a^\lambda=c^{l^m}$.
\end{proof}



\begin{proof}[Proof of \ref{vert}]
The image of $\rho_{\bar{a},l^ \infty }$ is a closed subgroup of $(\pi_1^{et})_l(\C^{\times})^r \cong \Z_l^r$ and is therefore a $\Z_l$-submodule. $\Z_l$ is a PID, so a $\Z_l$-submodule of the free module $\Z_l^r$ is free of dimension $s \leq r$. If $s \neq r$ then there are $\eta_1 ,...,\eta_r \in \Z_l$, not all zero such that
$$\eta_1 \rho_{a_1}(\sigma)+\cdots + \eta_r \rho_{a_r}(\sigma)=0$$
for all $\sigma$. Let $\eta_{j,m} \in \Z$ such that $\eta_{j,m}\equiv \eta_j \mod l^m$, and let
$$a=\eta_{1,m}a_1 + \cdots + \eta_{r,m}a_r.$$
Then $\rho_{a,l^m} = 0$ on $\Gal(K(\mu_{l^m},a^{1/l^m}) / K(\mu_{l^m} ))$ (since $\rho_{-,l^m}(-)$ as a function of two variables is bilinear) 
so by \ref{sahcor} $a^{\lambda} \in K^{\times^{l^m}}$, and this cannot happen for aribtrarily large $m$ since $K$ is finitely generated. So the image of $\rho_{\bar{a},l^ \infty }$ is an $r$-dimensional $\Z_l$-submodule of $\Z_l^r$. A $\Z_l$-submodule of $\Z_l$ is either 0, or is of the form $l^k \Z_l$ for some $k$, in which case it has index $l^k$ in $\Z_l$ and is isomorphic to $\Z_l$ and the result follows.
\end{proof}

Geometric homogenity follows by an identical argument, except we use the second statement of \ref{freeab} and there is no need to use Sah's lemma. 

\section{Grothendieck's section conjecture}

Taking into account the discussion in the introduction, it is interesting to draw analogies with Grothendieck's section conjecture of anabelian geometry. The conjecture involves hyperbolic curves i.e. projective algebraic curves of genus greater than one, minus finitely many points. A version of the section conjecture states that if $C$ is a hyperbolic curve over a finitely generated extension $K$ of $\Q$, then the profinite kummer map
$$\rho: C(K) \ra H^1(G_K , \pi_1^{et})$$
is a bijection. Here $\pi_1^{et}$ will be non-abelian so we are looking at non-abelian cohomology. However the construction of the map is the same, but with non-abelian $H^1$ we quotient cocycles by a right action of $\pi_1^{et}$
given by 
$$f.g = g^{-1} f(\sigma) \sigma(g)$$
where $f$ is a cocycle and $g \in \pi_1^{et}$. This is a generalization of the abelian $H^1$. \newline

So applied (wrongly) to this situation, the section conjecture would claim that the profinite Kummer map
\begin{align*}
 \rho: \G_m(K) &\ra H^1(G_K , \pi_1^{et})\\
 x & \mapsto \left[ \sigma \mapsto \frac{\sigma(x^{1/n})_n}{(x^{1/n})_n} \right]
\end{align*}
is a bijection. \newline

 The injectivity of the statement holds here: Suppose some $1 \neq x \in \G_m(K)$ is mapped to the class of trivial cocycle $\rho_x(\sigma)=[1]$. Then this means that $G_K$ does not conjugate any two elements in the fibre $\hat{p}^{-1}(x)$ at all. By Kummer theory (\ref{kummer}), this means that there $n$'th roots of $x$ in $K$ for arbitrarily large $n$ and this is impossible by \ref{freeab}. \newline

 On the other hand, the surjectivity fails since $\pi_1^{et} \cong \hat{\Z}$ is abelian and so $H^1(G_K, \pi_1^{et})$ is a $\hat{\Z}$-module, but $|\G_m(K)|=\aleph_0$. So there are elements of $H^1(G_K, \pi_1^{et})$ which come from fibres in the pro-\'etale cover about a rational point $x \in \G_m(K)$. \newline



\newpage

\end{document}